\documentclass[12pt]{amsart}
\usepackage{amsmath} %[tbtags]
\usepackage{amssymb} 
\usepackage{amsfonts}
\usepackage{ndstye}  
\usepackage{draftcopy}
\usepackage{comment}

\newcommand{\bc}{\mathbf C}

\newcommand{\br}{\mathbf R}

\newcommand{\bn}{\mathbf N}
\newcommand{\Cal}{\mathcal}

\renewcommand{\dim}{\operatorname{Dim}}

\newcommand{\id}{\operatorname{Id}}

\newcommand{\im}{\operatorname{Im}}

\newcommand{\ls}{\lesssim}

\newcommand{\mn}[1]{\Vert#1\Vert}

\newcommand{\mnorm}[1]{\Vert#1\Vert}

\newcommand{\ol}{\overline}

\newcommand{\re}{\operatorname{Re}}
\newcommand{\restr}[1]{\big|_{#1}}

\newcommand{\set}[1]{\left\{\,#1\,\right\}}

\newcommand{\w}[1]{\langle #1\rangle }

\newcommand{\wt}{\widetilde}

\numberwithin{equation}{section}

\newcommand{\bu}{{\bf u}}
\newcommand{\bm}{{\bf m}}

\newcommand{\bC}{\mathbb C}

\newcommand{\bM}{{\mathbb M}}

\begin{document}

%%% Sizes:

\textwidth 155 mm
\textheight 253 mm

\baselineskip 20pt %17pt
\lineskip 2pt
\lineskiplimit 2pt
\flushbottom

\textheight 258 mm % 255

\title[Symmetric Preparation]{Symmetric Preparation of Systems}
%\\[5mm]\boxed{\boxed{\text{\normalsize Version 3}}}}
\author[NILS DENCKER]{\protect\normalsize{\sc Nils Dencker\\ \\Lund University}}
\address{Department of Mathematics, University of Lund, Box 118,
S-221 00 Lund, Sweden}
\email{nils.dencker@gmail.com}
% \date\today

\maketitle

\section{Introduction and statement of results}\label{intro}

\thispagestyle{empty}

If $ f(t,x) $ is a  $ C^\infty $ function satisfying $  \partial_t^m f(0,0)  \ne 0 $ and $ \partial_t^j f(0,0) = 0$ for $ j < m $, then by the  Malgrange preparation theorem there exist  $ C^\infty $ functions $ c(t,x) $ and $ r_j(x) $, $ j < m$, such that $ c(0,0) \ne 0 $, $ r_j(0)= 0,\ \forall\, j$,  and
\begin{equation}\label{Wprep}
	f(t,x) = c(t,x) \left(t^m + \sum_{j < m} t^j r_j(x) \right ) \qquad (t,x) \in \br \times \br^n
\end{equation}
locally near $(0,0)$. In the analytic case, the Weierstrass preparation theorem, one has uniqueness but not in the $ C^\infty $  case, except for $ m = 1 $ and real valued $ f(t,x) $ when one can use the implicit function theorem. By taking the real part we find that real valued functions have real valued preparations.
There are also corresponding division results.

The Malgrange preparation theorem has been generalized to  $ N \times N $ systems of functions in  \cite{de:prep}. An example is for smooth  $ N \times N $ system $ F(t,x) $ satisfying  
\begin{equation}\label{Fcond}
| \partial_t^m F(0,0) | > 0 \quad \text{and} \quad \partial_t^j F(0,0) = 0, \quad j < m
\end{equation}
Then there exist  smooth  $ N \times N $ systems $ C(t,x) $ and $ M_j(x) $,  $j < m$,  such that $ | C(0,0) | \ne 0 $, $ M_j(0)= 0 $, $\forall \, j$, and
\begin{equation}\label{Wprepsys}
F(t,x) = C(t,x)\left(t^m\id_N + \sum_{j =0}^{m-1} t^j M_j(x) \right ) \qquad (t,x) \in \br \times \br^n
\end{equation}
locally near $(0,0)$, this follows from \cite[Th.\ 5.3]{de:prep}. For analytic systems one has uniqueness but not for smooth, and real valued systems have real valued preparations.

We also proved the following division theorem in \cite{de:prep}.
If $ F(t,x) $ is a smooth  $ N \times N $ system satisfying~\eqref{Fcond}, then for any smooth  $ N \times N $ system $ G(t,x) $ there exist smooth  $ N \times N $ systems $ C(t,x) $ and $ R_j(x) $, $j < m$, such that
\begin{equation}\label{Wdiv}
G(t,x) = C(t,x) F(t,x) + \sum_{j =0}^{m-1} t^j R_j(x) \qquad (t,x) \in \br \times \br^n
\end{equation}
locally near $(0,0)$, this is a special case of \cite[Th.\ 5.9]{de:prep}.  One has uniqueness in the analytic case but not in the smooth case, and real valued systems have real valued divisions. By taking the adjoints, we also obtain adjoint preparation and division results, see \cite{de:prep}.

A more general case is for smooth  $ N \times N $ system $ F(t,x) $ satisfying  
\begin{equation}\label{Fcond0}
	\partial_t^m \det F(0,0)   \ne 0 \quad \text{and} \quad \partial_t^j \det F(0,0) = 0, \quad j < m
\end{equation}
Then by \cite[Th.\ 5.3]{de:prep} we may factor
\begin{equation}\label{Wprepsys0}
	F(t,x) = C(t,x)\left(\sum_{j=0}^{n}t^j \pi_j + M(t,x)\right ) \qquad (t,x) \in \br \times \br^d
\end{equation}
locally near $(0,0)$. Here $\pi_j$ is an orthogonal projection in $\bC^N$, $j = 0,\dots, n$ such that $\pi_j\pi_k = \delta_{jk} \pi_k$ and $\sum_{j=0}^{n}\pi_j = \id_{N}$ but  some projections could be on the zero space $\{ 0\}$ so that $\pi_j = 0$. We have that $ C(t,x) $ and $ M(t,x)  = \sum_{j =0}^{n} t^j M_j(x) $ are  smooth  $ N \times N $ systems such that $ | C(0,0) | \ne 0 $, $M_j(x) \pi_k \equiv 0,\ \forall \, j \ge k$, and
\begin{equation}\label{Mcond0}
	M_j(0) = \sum_{i < j < k} \pi_iM_j(0)\pi_k \qquad  \forall \, j
\end{equation}
which implies that $M_0(0) = 0$. The projections $\pi_j$ and matrices $M_j(0)$ are uniquely determined by~\eqref{Mcond0}.
In the analytic case, the preparation~\eqref{Wprepsys0} is unique but not in the smooth case.
Observe that $F(t,x) \pi_m = C(t,x)\big (t^m\pi_m + \sum_{j < m} t^j M_j(x)\pi_m\big )$, so for the case when $\pi_m = \id_N$ we find from~\eqref{Mcond0} that $M_j(0) = 0$, $j < m$. 

We also proved the following division theorem in  \cite[Th.\ 5.9]{de:prep}.
If $ F(t,x) $ is a smooth  $ N \times N $ system satisfying~\eqref{Wprepsys}, then for any smooth  $ N \times N $ system $ G(t,x) $ there exist smooth  $ N \times N $ systems $ C(t,x) $ and $ R(t,x)  = \sum_{j =0}^{n-1} t^j R_j(x) $ such that $R_j(x) \pi_k \equiv 0$ when $j \ge k$ for the projections in~\eqref{Wprepsys0} and
\begin{equation}\label{Wdiv0}
	G(t,x) = C(t,x) F(t,x) + R(t,x) \qquad (t,x) \in \br \times \br^n
\end{equation}
locally near $(0,0)$. As before, one has uniqueness in the analytic case but not in the smooth  case. 
This gives 
$$
G(t,x) \pi_m = C(t,x)\left (t^m \pi_m + \sum_{j < m} t^j M_j(x)  \pi_m \right) + \sum_{j < m} t^j R_j(x)  \pi_m
$$
By taking the adjoints, we also obtain adjoint preparation results and corresponding division results, see \cite{de:prep}.

Now, the question about symmetric preparation results was raised by Gregory Berkolaiko \cite{GB}. The simplest case is for  analytic (or smooth) symmetric $ N \times N $ matrices $ F(t,x) $  satisfying $  \partial_t F(0,0) > 0 $ and $ F(0,0) = 0$. Then the question was: do there exist analytic (or  smooth) $ N \times N $ matrices  $ U(t,x) $ and $ M(x) $, such that  $| U(0,0) | \ne  0$,  $ M(x) $ is symmetric with $ M(0) = 0$, satisfying
\begin{equation}\label{sprepconj}
	F(t,x) = U(t,x)\left(t \id_N + M(x)\right )U^*(t,x) \qquad (t,x) \in \br \times \br^n
\end{equation}
near $ (0,0) $, and is this preparation unique? Observe that the conditions on $ F(t,x) $ are necessary for~\eqref{sprepconj} since  $\partial_t F = UU^* > 0 $ and $ F = UMU^* = 0$ at $(0,0)$.

Here and in the following we say that the matrix $ A $ is symmetric if $ A^* = A $ and antisymmetric if $ A^* = -A $, where $ A^* = \ol A^{t} $, and when $A > 0$ we assume that $A$ is symmetric.
The result~\eqref{sprepconj} is true even for higher order vanishing, and when $U(t,x) > 0$ the preparation is unique, see Theorem~\ref{anasymprep}. We shall generalize this preparation to the finite order vanishing case without the positivity condition.

\begin{rem}\label{Rem1.1}
	The preparation~\eqref{sprepconj} is not unique in general. For example, if $ U(t,x) $ and  $\bM(t,x) $ are solutions to~\eqref{sprepconj}, then  for any orthogonal $ A(x) = (A^{-1})^{*}(x)$ we find that $\wt U(t,x) = U(t,x)A(x) $ and $\wt \bM(t,x) = A^{*}(x) \bM(t,x) A(x) $ also solves~\eqref{sprepconj}. 
	But if $U(t,x) > 0$ and $UA(t,x) > 0$ then  $A(t,x) \equiv \id_N$ and the preparation~\eqref{sprepconj} is unique.
\end{rem}

In fact, if $U > 0$ then $U$ has a ON base of eigenvectors, and if $A$ is  orthogonal and $UA$ is symmetric then $A^*U = UA$ which gives $U = AUA$. Then $A$ preserves the eigenspace with the largest eigenvalue $\mn{U}$ and $A^2 = \id$ on this space. By iteration over decreasing eigenvalues of $U$ we find $A^2 = \id_N$. Then $A = A^{-1} = A^*$ has eigenvalues $\pm 1$, and if  $ UA > 0$ there cannot be any negative  eigenvalues, so we obtain that $A = \id_{N}$. In fact, $Au = -u$ gives $\w{UAu,u} = - \w{Uu,u} < 0$, which gives a contradiction. To formulate the results of the paper, we shall need the following definitions.

\begin{defn}
We say that the set $\pi = \{ \pi_j \}_j$ of orthogonal projections in $\bc^N$ is a complete set if $\pi_j \pi_k = 0$ when $j \ne k$ and $\sum_j \pi_j = \id_N$, thus $\id_N = \oplus_j \pi_j$. If $A$ is an $N \times N$ system, we say that $A $ is completely decoupled with respect to $\pi$ if $A = \sum_j \pi_j A \pi_j$.
\end{defn}

An example is the projections in \eqref{Wprepsys0}. The following is the main result of the paper.

\begin{thm}\label{anasymprepgen}
	Assume that $ F(t,x) $ is a smooth symmetric $ N \times N $ matrix valued function of  $ (t,x) \in \br \times \br^n$ so that there exists $m \ge 0$ such that
	\begin{equation}\label{Fcond1}
		\partial_t^m \det F(0,0)   \ne 0 \quad \text{and} \quad \partial_t^j \det F(0,0) = 0, \quad j < m
	\end{equation}
Then there exist a neighborhood $ \omega $ of\/ $ (0,0)$, a complete set of constant orthogonal projections $ \big\{ \Pi_\pm^j  \big\}_{j, \pm}$ and smooth  $ N \times N $ matrix valued functions $ U(t,x)  $, $O(t,x)$ and\/  $ \bM(t,x)  $ such that $  U$ is invertible, $O$ is orthogonal  and $ \bM (t,x)$ is symmetric so that
	\begin{equation}\label{sprepresgen}
		F(t,x) = O(t,x)U(t,x)\left(\sum_{j, \pm} \pm t^j \Pi_{\pm}^j + \bM(t,x)\right )U^*(t,x) O^*(t,x) \quad \text{in $ \omega $}
	\end{equation}
	Here $U(t,x) $ and   $\bM(t,x)$ are completely decoupled with respect to $  \big\{\Pi_{\pm}^j  \big\}$, $\bM(t,0) \equiv 0$, $\Pi_{\pm}^j \bM(t,x)\Pi_{\pm}^j $ is a polynomial in $t$ of degree $< j$ and  $O(t,x)$ maps\/ $\oplus_{\pm}^j \bc^{d_\pm^j}$ to\/ $\oplus_{\pm}^j V_\pm^j$ where $V_\pm^j = \im \Pi_\pm^j$ and $d_\pm^j = \dim V_\pm^j$, so $m = \sum_{j, \pm} j\cdot d^j_\pm$.
	If $ F(t,x) $ is analytic (or formal power series) then we obtain \eqref{sprepresgen} so that $ U(t,x) $, $O(t,x)$ and $ \bM (t,x) $ are analytic  (or formal power series). We may choose \/ $ U > 0 $ and then the preparation is unique.
	\end{thm}
	
	Here the case $m = 0$ means that $F(0,0)$ is invertible.
	Note that $U > 0$ is equivalent to $U_\pm^j = \Pi_\pm^j U_\pm^j \Pi_\pm^j > 0$ on $V^\pm_j $ if $\Pi_\pm^j \ne 0$.

\section{The reduction to normal forms}\label{normal}

Since the determinant is the product of the eigenvalues, we find from condition~\eqref{Fcond1} that the eigenvalues of $F(t,0)$ vanish of finite order at $t=0$.
Let $F_0(t,x) = F(t,x)$ and 
\begin{equation}\label{vanishingdef}
	\mu = \inf \big\{ j:\ \partial_t^j F_0(0,0) \ne 0 \big\} \ge 0
\end{equation}
then the nonzero eigenvalues $\lambda_j$ of $\partial_t^\mu F_0(0,0) $ satisfy $\min_j\big\{| \lambda_j | \big\} = c_0 > 0$. Taking $\mu_0 = \mu$, the ON projection $ \Pi^{\mu_0}_\pm(t,x)$ on the eigenvectors such that $\pm \lambda_j(t,x) > c_0/2$ and $V_\pm^{\mu_0} = \im \Pi_\pm^{\mu_0}$ near the origin, we obtain that $\Pi^{\mu_0}_\pm F_0 \Pi^{\mu_0}_\mp \equiv \Pi^{\mu_0}_\pm F_0 \big(\id_N - \Pi^{\mu_0}_\pm \big ) \equiv 0$ and
\begin{equation}\label{Fmupm}
F_{\pm}^{\mu_0}(t,x) = \Pi_{\pm}^{\mu_0}(t,x) F_0(t,x) \Pi_{\pm}^{\mu_0}(t,x)
\end{equation}
satisfies $\pm  \partial_t^{\mu_0} F_{\pm}^{\mu_0}(0,0) = \Pi_{\pm}^{\mu_0}  \partial_t^{\mu_0} F_0 \Pi_{\pm}^{\mu_0}(0,0)\ge c_0$ on $ V_\pm^{\mu_0}(0,0)$ if $ \Pi_{\pm}^{\mu_0} \ne 0 $ by~\eqref{vanishingdef}.

If $\Pi^{\mu_0} =  \Pi^{\mu_0}_- +  \Pi^{\mu_0}_+ $, the projection on all the eigenvalues, then $\Pi^{\mu_0} F_0 \big(\id_N - \Pi^{\mu_0} \big ) \equiv 0$, so letting 
$$
F_1(t,x) = (\id_N - \Pi^{\mu_0}(t,x)) F_0(t,x) (\id_N - \Pi^{\mu_0}(t,x))
$$ 
we find that $\partial_t^j F_1(0,0) = 0$ for $j \le \mu_0$. Then applying~\eqref{vanishingdef} on $F_1 $ gives $\mu = \mu_1 > \mu_0$ and ON projections $\Pi^{\mu_1}_\pm(t,x) $ on the positive and negative eigenvalues of $\partial_t^{\mu_1} F_1 $. Then $\Pi^{\mu_1}_\pm F_1 \Pi^{\mu_1}_\mp \equiv \Pi^{\mu_1}_\pm F_1 \big(\id_N - \Pi^{\mu_1}_\pm \big ) \equiv 0$. We obtain $ F^{\mu_1}_\pm$ from~\eqref{Fmupm} with $F_0$ and  $\mu_0$  replaced by $F_1$ and $\mu_1$, 
then $\pm  \partial_t^{\mu_1} F_{\pm}^{\mu_1}(0,0) = \Pi_{\pm}^{\mu_1}  \partial_t^{\mu_1} F_1 \Pi_{\pm}^{\mu_1}(0,0)\ge c_1 > 0$ on $ V_\pm^{\mu_1}(0,0)$ if $ \Pi_{\pm}^{\mu_1} \ne 0 $.

We may iterate this until $ \big\{\Pi_{\pm}^{\mu_j}(t,x)  \big\}$ is a complete set of  orthogonal spectral projections $\Pi_{\pm}^{\mu_j}(t,x) $
for increasing $\mu_j$. We find that $F = \sum_{j,\, \pm} F_{\pm}^{\mu_j}$, with $F_{\pm}^{\mu_j} $ given by~\eqref{Fmupm} with  $F_0$ and $\mu_0$ replaced by $F_j$  and $\mu_j$, is completely decoupled with respect to $\big\{\Pi_{\pm}^{\mu_j}(t,x)  \big\}$. We find that  $\pm  \partial_t^{\mu_j} F_{\pm}^{\mu_j}(0,0) \ge c_{j} > 0$ on $V_\pm^{\mu_j}(0,0) $ if $ \Pi_{\pm}^{\mu_j} \ne 0$ and $ \partial_t^{\ell} F_{\pm}^{\mu_j}(0,0) = 0$ when $\ell < \mu_j$. Note that analytic (or formal power series) systems have analytic  (or formal power series) projections.
We shall obtain Theorem \ref{anasymprepgen} from the following result.

\begin{thm}\label{anasymprep}
Assume that $ F(t,x) $ is a smooth symmetric  $ N \times N $ matrix valued function of  $ (t,x) \in \br \times \br^n$ such that 
\begin{equation}\label{Fcond2}
\partial_t^m F(0,0) > 0 \quad \text{and}  \quad \partial_t^jF(0,0) = 0  \quad \forall\, j < m
\end{equation}
for some $m \ge 0$. Then there exists a neighborhood $ \omega $ of\/ $ (0,0)$ and  smooth  $ N \times N $ matrix valued functions $ U(t,x)  $ and\/  $M_j(x)  $, $ j < m$, such that $ | U(0,0) | \ne 0 $, $ M_j(x)$ is symmetric such that $ M_j(0) = 0,\ \forall \, j $, and
\begin{equation}\label{sprepres}
F(t,x) = U(t,x)\left(t^m \id_N + \sum_{j=0}^{m-1} t^j  M_j(x)\right )U^*(t,x) \qquad \text{in $ \omega $}
\end{equation}
If $ F(t,x) $ is analytic (or  formal power series) then   $ U(t,x) $ and $ M_j (x) $ are analytic (or  formal power series).  We may choose $ U > 0 $ and then the preparation is unique.
\end{thm}

Note that the uniqueness does not hold in general, since Remark~\ref{Rem1.1} gives counterexamples for nonsymmetric $ U $ in these cases also.
Observe that \eqref{sprepres} implies $\partial_t^j F = j! UM_jU^* = 0 $ for $ j < m $, and $ \partial_t^m F = m!UU^* > 0 $  at $(0,0)$, so the conditions are necessary. 
The case $m = 0$ means that $F(0,0) > 0$ so $F$ is invertible in a neighborhood of $(0,0)$. 

\begin{proof}[Proof of Theorem \ref{anasymprepgen}]
	By letting $k = \mu_j $ we obtain a smooth orthogonal base for $ V_{\pm}^k(t,x) $ near $(0,0)$ by taking an orthogonal base for $ V_{\pm}^k(0,0) $, applying  $ \Pi_{\pm}^k(t,x) $ on that base to get a base for  $ V_{\pm}^k(t,x) $ near $(0,0)$ and then using Gram--Schmidt on this base. This orthogonal base is analytic (or  formal power series) if $\Pi^k_\pm(t,x) $ is. This gives smooth ON mappings $ O_{\pm}^k(t,x)$ from $\bc^{d_{\pm}^k}$ to $V_{\pm}^k(t,x) $, $d_{\pm}^{k} = \dim V_\pm^{k}$, with the adjoints as inverses. Then 
	$$
	F_{\pm}^k(t,x) = O_{\pm}^k(t,x)\Pi_{\pm}^k(t,x) F(t,x)\Pi_{\pm}^k(t,x) (O_{\pm}^k(t,x))^* 
	$$ 
	is a smooth ${d_{\pm}^k} \times {d_{\pm}^k}$ system. Since $ \partial_t^j F_{\pm}^k = (O_{\pm}^k)^*\Pi_{\pm}^k (\partial_t^j F)\Pi_{\pm}^k O_{\pm}^k$ at $(0,0)$ for $j \le k$, Theorem~\ref{anasymprep} gives $\pm F_{\pm}^k(t,x) = U_{\pm}^k(t,x)\big(t^k \id_{d_{\pm}^k} +  \bM_{\pm}^k(t,x)\big)(U_{\pm}^k)^*(t,x) $ which gives
	\begin{multline}
	 \Pi_{\pm}^k(t,x) F(t,x)\Pi_{\pm}^k(t,x) \\ = O_{\pm}^k(t,x)U_{\pm}^k(t,x)\left(\pm t^k \id_{d_{\pm}^k} \pm  \bM_{\pm}^k(t,x)\right)(U_{\pm}^k)^*(t,x)(O_{\pm}^k(t,x))^* \quad \text{on $V_\pm^{k}$}
	\end{multline}
	Here $U^k_{\pm}(t,x)$ and $\bM_{\pm}^k(t,x)$ are smooth ${d_{\pm}^k} \times {d_{\pm}^k}$ systems, $|U_{\pm}^k(0,0)| \ne 0$ and $\bM_{\pm}^k(t,x)$ is a symmetric polynomial system of degree less than $k$ such that $\bM_{\pm}^k(t,0) \equiv 0$. For analytic systems we obtain that this preparation  is  analytic and similarly for formal power series.
	
	We obtain Theorem~\ref{anasymprepgen} by summing up and letting $O(t,x) = \oplus_{\pm}^j O_{\pm}^j(t,x) $, which maps $\bc^N = \oplus_{\pm}^j \bc^{d_{\pm}^j}$ into $\oplus_{\pm}^j V_{\pm}^j(t,x) $, $U(t,x) = \oplus_{\pm}^j U_{\pm}^j(t,x) $ and $\bM(t,x) = \oplus_{\pm}^j \pm\bM_{\pm}^j(t,x) $, which are completely decoupled with respect to  $ \big\{ \Pi_{\pm}^j  \big\}_{j, \pm}$ where $\Pi_{\pm}^j = \Pi_{\pm}^j(0,0)$.
\end{proof}

\begin{rem}
If $\partial^j_t\det F(0,0) = 0$ for all $ j$, then the preparation of $F$ would stop with a remainder space $V^\infty = \im \Pi^\infty \ne {0}$, with $\Pi^\infty = \id_N - \oplus_{\pm}^j \Pi_{\pm}^j, $ that is orthogonal to all\/ $V^j_\pm$ so that $\partial^j_t F\restr{V^\infty}(0,0) = 0$ for all $j$. Adding a term $\varepsilon\Pi^\infty$ with $\varepsilon \ne 0 $ to $ F$ then makes $t \mapsto \det F(t,0)$ vanish of finite order $t = 0$ since $\Pi^0 = \Pi_{-}^0 + \Pi_{+}^0$ changes to $ \Pi^0 + \varepsilon\Pi^\infty$.
\end{rem}

If $\pm \partial_t^m F(0,0) > 0$ but $F$ does {\em not}\/ satisfy the other conditions in \eqref{Fcond2} in Theorem~\ref{anasymprep}, then by subtracting  $R(t) = \sum_{j=0}^{m-1} t^j \partial_t^j F(0,0)/j! $, we obtain the following result.

\begin{cor}
If $F(t,x)$ is smooth symmetric  $ N \times N $ matrix valued function of \/ $ (t,x) \in \br \times \br^n$ such that  $\pm \partial_t^m F(0,0) > 0$, then  we obtain that
\begin{equation}\label{spreprescor0}
F(t,x) = U(t,x)\left(\pm t^m \id_N +  \sum_{j=0}^{m-1} t^j  M_j(x)\right )U^*(t,x) + \sum_{j=0}^{m-1}  R_j t^j/j!  \quad \text{near $ (0,0)$}
\end{equation}
where $U(t,x)$ is invertible, $M_j(x)$ is symmetric, $M_j(0) = 0$ and $R_j = \partial_t^j F(0,0)$ for $j < m$.
We can choose $U > 0$ and then the preparation is unique, and
if $F$ is analytic (or  formal power series), then $ U$ and $M_j$ are analytic (or  formal power series).
\end{cor}

Theorem \ref{anasymprep} will be proved in Section~\ref{analytic} for the analytic case, and in Section~\ref{smooth} for the smooth  case. But first we shall consider formal power series case.

\section{Formal power series solutions}\label{formal}

We shall assume that  $ F(t,x) $ in Theorem~\ref{anasymprepgen} is given by a formal power series expansion, then it is easy to see that the reduction to Theorem~\ref{anasymprep} is also given by formal power expansions. To prove Theorem~\ref{anasymprep} we take the expansion
\begin{equation}
	F(t,x) =  \sum_{{j, k}}  F_{j,k}t^jx^k
\end{equation}
where $j$ and $k \in \bn$, $F_{j,0} = \partial_t^jF(0,0)/j! = 0$ for $ j < m $ and $  F_{m,0}  =  \partial_t^m F(0,0)/m!  > 0 $. Here we use the tensor notation $F_{j,k} = \set{\partial_t^j \partial_{x}^\alpha F/j!\alpha!}_{|\alpha | = k}$ with $\alpha ! = \prod_j \alpha_j !$ and $x^k = \set{x^{\alpha}}_{|\alpha | = k}$.
We take the expansions $ U(t,x)   =   \sum_{{j,k}}  U_{\substack{j,k}} t^j x^k$ where $| U_{0,0}| \ne 0 $ and  
\begin{equation}\label{Mexp}
	\bM(t,x)   =   \sum_{j < m} t^j \bM_j(x)  = \sum_{\substack{j < m\\k \ne  0}} M_{j,k}t^j x^k
\end{equation} 
is symmetric such that $ \bM_{j}(0) = M_{j,0} = 0 $, $ j < m $. 

Then \eqref{sprepres} gives 
\begin{equation}\label{Fexp}
	\sum_{j, k} F_{j, k}  t^j x^k
	= \sum_{j, k}U_{j, k} t^j x^k \Big( t^{m}\id_N  + \sum_{\substack{j,k}} M_{j,k}t^j x^k \Big)	
		\sum_{{j,k}} U_{j,k}^* t^j x^k
\end{equation}
and by matching coefficients we obtain that
\begin{equation}\label{Fjkexp}
	F_{j,k} = \sum_{i, \ell} U_{i, \ell } U_{j - m - i,k - \ell }^* + \sum_{\substack{i,\ell, p, q}} U_{i, \ell } M_{p,q}U_{j - p -i,k - q -\ell}^*
\end{equation}
where the sum is over the nonnegative indices. The first terms of this expansion are $ F_{0,0} = U_{0,0}M_{0,0}U_{0,0}^* = 0 $,  $F_{0,1} = U_{0,0}M_{0,1}U_{0,0}^*$, $F_{j,0} = 0$, $ j < m $, and  $ F_{m,0} = U_{0,0}  U_{0,0}^* $. Here we can choose $U_{0,0} =  ( F_{m,0})^{1/2} $ which is the unique positive definite solution. In fact, since $U_{0,0}$ is invertible we can write any other positive symmetric solution as $U_{0,0}A$ and then $(U_{0,0}A)^* = A^*U_{0,0} = U_{0,0}A$. Since $F_{m,0} = (U_{0,0}A)^2 = U_{0,0}AA^*U_{0,0} = U_{0,0}^2$, $A$ has to be orthogonal. But then Remark~\ref{Rem1.1} gives that $A = \id_N $ which gives us a unique  $U_{0,0} > 0 $.

When $ j < m $ we obtain from \eqref{Fjkexp} that
\begin{equation}\label{mexpk}
	U_{0,0}M_{j, k}U_{0,0}^* = F_{j,k} - \sum_{\substack{i + \ell < j + k}} U_{i, \ell } M_{p,q}U_{j - p -i,k - q -\ell}^*
\end{equation} 
Since $ |U_{0,0} | \ne 0 $  we obtain symmetric $ M_{j,k}$ that is given by $ F_{j,k} $, $ U_{i,\ell} $ with $ i \le j $ and $ \ell \le k $, and by $ M_{p,q} $ with $p \le j$, $ q \le k $ and $ p + q < j + k $.
Thus  $ M_{j,k}$ is determined by $ F_{j,k} $, $ U_{i,\ell} $ with $ \ell < k$ or $\ell = k$ and $i <j$ and by $ M_{p,q} $ with $ q < k$ or $q = k$ and $p <j $.

With $ j $ replaced by $ j + m $, we also obtain from \eqref{Fjkexp} that 
\begin{equation}\label{jkrec}
	2\re U_{j,k} U_{0,0}^* = F_{j+m,k} - \sum_{\substack{0 < i + \ell < j + k}} U_{i, \ell} U_{j - i,k - \ell }^* 
	- \sum_{q \ne  0} U_{i, \ell } M_{p,q}U_{j + m - p - i,k - q - \ell}^*
\end{equation}
Here the right hand side is given by $ F_{j+m,k} $, 
$ U_{i,\ell} $ with $ i \le j + m$ and  $ \ell < k $ or with $ i + \ell < j + k $, thus by $ \ell < k$ or $\ell = k$ and $i <j $, and  $ M_{p,q} $ with $ p < m $ and $q \le k $.  Since $U_{0,0}$ is invertible we can solve~\eqref{jkrec}, and in the case $U_{0,0} > 0 $ we find that Lemma~\ref{symlemma}  gives a unique symmetric $U_{j,k}$.

We can now determine $ U_{j,k} $ and $ M_{j,k} $  from \eqref{mexpk} and \eqref{jkrec} by induction over $j \ge 0$ for fixed $ k \ge 0$ starting with $ U_{0,k} $ and $ M_{0,k} $ which we first determine by induction over $k$.
Since $M_{0,0} = 0$ we start with $U_{0,0}$ and $M_{0,1} = U_{0,0}^{-1} F_{0,1}( U_{0,0}^*)^{-1}$, and assume that $ F_{j,k}$ is known. We also assume that $ U_{0,\ell} $ and $ M_{0,\ell + 1} $ is known for $\ell < k$. Then $ U_{0,k} $ is determined by these and then $ M_{0,k+1} $ is determined by $U_{0,\ell}$ and $ M_{0,\ell } $ for $ \ell \le k$. By induction we obtain $ U_{0,k} $ and $ M_{0,k+1} $ for any $k \ge 0$.

Then we use induction over $j \ge 0$ for $ U_{j,k} $ and $ M_{j,k+1} $ starting with $k = 0$. We assume that we have determined $ U_{i,\ell} $ and $ M_{i,\ell+1} $ when $\ell < k$ and when $\ell = k$ for $i < j$.  Then $ U_{j,k} $ is determined by these and we then obtain $ M_{j,k+1} $ from $U_{i,\ell}$  when $ \ell \le  k$ and from $ M_{i,\ell +1} $ when $ \ell < k$ and when  $\ell = k $ for $i < j$. By induction we obtain $ U_{j,k} $ and $ M_{j,k+1} $ for any $j, \ k $. If $U_{0,0} > 0$, then any finite expansion of $U(t,x)$ is positive definite for $| t | + | x | \ll 1$.

\begin{prop}\label{powerres}
	If $ F(t,x) $ is symmetric and given by a formal power series such that $ \partial_t^m F(0,0) > 0 $ and $ \partial_t^jF(0,0) = 0$, $ \forall\, j < m $, then there exists formal power series $ U(t,x) $ and $ \bM(t,x) $ given by~\eqref{Mexp} solving \eqref{sprepres} modulo $ O(|t |^N + | x |^N) $, $ \forall \, N $.
	We get a unique symmetric power series solution in the case when $ U_{0,0} > 0$.
\end{prop}

But even in the analytic case, it is not clear that the formal power expansions converge to analytic functions, except in the twodimensional case, see Remark~\ref{Rem3.2}. We shall instead use inverse function theorems, see Sections~\ref{analytic} and~~\ref{smooth}.

\begin{rem}\label{Rem3.2}
If \/ $ n = 1 $, $ m = 1 $ and $ F $ is analytic, then the nested Artin approximation {\rm (see e.g.\ Theorem 5.18 in \cite{rond})} gives analytic $\ol U(t,x) = \sum_{j,k} \ol U_{jk} t^jx^k$  and $\ol M(x) = \sum_k \ol M_k x^k = \ol M^*(x)$ that solves
$$
F(t,x) = \ol U(t,x)(t Id + \ol M(x))\ol U^*(t,x) \qquad (t,x) \in \br^2
$$
which gives Theorem \ref{anasymprep} in the two dimensional analytic case. One also obtains that any finite number of terms in the expansions of $\ol U(t,x)$ and $\ol M(x) $ coincide with the corresponding terms in formal power series solutions $U(t,x) $ and $ M(x) $. 
But when  $ n > 1 $ this does not work in general, see the counterexamples in \cite{rond}. 
\end{rem}

\section{Proof of Theorem \ref{anasymprep} in the analytic case}\label{analytic}

In this section we shall prove the analytic case, by reducing to the case where the functions are holomorphic in an open polydisc in $\bC^n$ and continuous on the boundary of the polydisc. In fact, any analytic function has a Taylor expansion that converges in an open polydisc to a holomorphic function, so we obtain this in any smaller closed polydisc. 

The $ N \times N $ systems that are holomorphic in an open polydisc in $\bC^n$ and continuous on the closed polydisc form a Banach space with the usual $ L^{\infty} $ norm on the closed polydisc. We shall use the inverse function theorem on these Banach spaces, when the differential is surjective (or bijective when $U>0$), giving surjective (or bijective) mappings.

We shall show for any  $ N \times N $ system $F(t,x)$ that is holomorphic in a polydisc $| z_j | < \varepsilon $  with radius $\varepsilon > 0$ for $1 \le j \le n + 1$ and continuous on the boundary, there exists $\delta > 0$ so that~\eqref{sprepres} holds in the polydisc with radius $ \delta$.
This means that if $ U(t,x) $ and $ \bM(t,x) $ are analytic such that  $ | U(0,0) | \ne 0$,  $ \bM (t,x) = \sum_{j=0}^{m-1} t^j  M_j(x)/j!$ is symmetric polynomial such that $ M_j(0) = 0,\ \forall \, j $, then the mapping
\begin{equation}\label{3.1}
(U(t,x), \bM(t,x)) \overset{\Cal F}{\longrightarrow} \Cal F(U,\bM) = U(t,x)(t^m \id_N + \bM(t,x))U^*(t,x) = F(t,x)
\end{equation}
is analytic and surjective in some neighborhood of $ (0,0) $ for analytic and symmetric $ F(t,x) $ satisfying $\partial_t^m F(0,0) > 0 $ and $\partial_t^j F(0,0) =  0$ for $j < m$. 

When $ \bM \equiv 0 $ we find that $  \partial_t^m(t^mUU^*) = m!UU^*= \partial_t^mF $ when $ t = 0 $ so $ U(t,x) $ must be invertible close to $ (0,0) $. We may choose $U(0,0) = \partial_t^mF(0,0)^{1/2} > 0$ and then $U > 0$ near the origin. For such $ U(t,x) $, the differential of $ \Cal F $ at $ (U,0) $ is given by
\begin{multline}\label{3.2}
({\bf u}(t,x), {\bm}(t,x)) \overset{d\Cal F}{\longrightarrow} d \Cal F({\bf u(t,x)},{\bf m(t,x)})\\ = 2t^m\re U(t,x){\bf u}^*(t,x) + U(t,x){\bm}(t,x) U^*(t,x)  = A(t,x)
\end{multline}
where $ \bu(t,x) $ is analytic, $ \bm(t,x) $ and $A(t,x)$ are analytic and symmetric such that $ \bm $ is polynomial of degree less than $m$.
We shall show that this differential is surjective for such systems in a neighborhood of~$ (0,0) $. This will give surjectivity of~\eqref{3.1} for analytic $F$ close to $t^mUU^*$ near $(0,0)$.
In order to study the mapping~\eqref{3.2} we shall use the Taylor operator
\begin{equation}\label{Top}
T_k(A)(t,x) = \sum_{j=0}^{k-1} t^j  \partial_t^jA(0,x)/j!
\end{equation}

\begin{rem}\label{Uniqm}
We find from \eqref{3.2} that $ \bm = T_m(U^{-1}A (U^{-1})^*) $ is the unique polynomial of degree less than $m$ which satisfies $ T_m(U\bm  U^*) =T_m(A)$. In fact, the mapping $\bm \mapsto  T_m(U\bm  U^*)$ is triangular in the Taylor expansion with invertible diagonal $U(0,0)U^*(0,0)$. If $A$ is invertible at the origin then $\bm \mapsto T_m(A\bm)$ is a bijection, since it is upper triangular in the Taylor expansion with diagonal $A(0)$.
\end{rem}

We then find that 
\begin{equation*}
2t^m\re U(t,x)\bu^*(t,x)  = A(t,x) -  T_m(U\bm  U^*)(t,x)
\end{equation*}
and since the right hand side vanishes of order $m$ at $t = 0$ we obtain that
\begin{equation*}
2 \re U(t,x) \bu^*(t,x) = t^{-m} \Big(A(t,x) -  T_m(U\bm  U^*)(t,x)\Big) = A_0(t,x)
\end{equation*}
which is analytic and symmetric. We find that $ 2 U  \bu^* = A_0 + B_0 $ for some analytic and antisymmetric~$B_0$, which gives  $\bu^* = U^{-1}(A_0 + B_0)/2$ and the solution
\begin{equation}\label{ansymsol}
\bu(t,x) = (A_0(t,x) - B_0(t,x))(U^{-1})^*(t,x)/2
\end{equation}
This gives the surjectivity of $ d\Cal F $ in~\eqref{3.2} and by the inverse function theorem we obtain surjectivity of the nonlinear mapping $ \Cal F $ in~\eqref{3.1} in a neighborhood $ V $ of $(U,0) $ near $(0,0)$. 

If $F(t,x)$ satisfies \eqref{Fcond2} then for any $C^\infty$ neighborhood $V$ of $\Cal F(U,0) = t^m UU^*$ there exists a neighborhood $\omega$ of $(0,0)$ so that $F \in V$ in $\omega$. By shrinking the Banach space to a smaller closed polydisc we obtain a solution to \eqref{sprepres}. Since $\partial_t^j F = j! UM_jU^* = 0$ at $(0,0)$ for $j < m$ we find that $M_j(0)= 0$ which gives $\bM(t,0) \equiv 0$.
Observe that $\im  \Cal F $ will also contain some analytic $ F(t,x) $ that does not vanish at $(0,0)$, but then $\bM(0,0) \ne 0$.

But \eqref{ansymsol} is not unique, since any antisymmetric $ B_0 $ gives a solution. Actually, we have for fixed $ U $ and $ \bM \equiv 0 $ that the mapping $ (A,B_0) \mapsto (\bu, \bm) $ is bijective.
So it remains to prove uniqueness in the case $ U(t,x) $ is positive definite near $ (0,0) $, since then $U(0,0) =(\partial_t^m F(0,0)/m!)^{1/2} > 0 $ by Remark~\ref{Rem1.1}. Since $ \bm $  is uniquely determined by $U$ and $A$, we only have to prove that there is a unique symmetric solution $ \bu $ to~\eqref{ansymsol}.
If $\im \bu = 0$ then we obtain from Lemma~\ref{symlemma} a unique analytic antisymmetric  $ B_0(t,x)$ that solves
$$
\im (U^{-1} B_0) = - \im (U^{-1} A_0) \qquad \text{near~$ (0,0) $}
$$ 
for the analytic symmetric $ A_0(t,x)$ in \eqref{ansymsol}. Thus the unique analytic symmetric solution is given by 
$$ 
\bu(t,x) =   U^{-1}(t,x)(A_0(t,x) + B_0(t,x))/2  
$$ 
We then find that $ d \Cal F $ is bijective at $(U,0)$ which gives that $ \Cal F $ is a bijection near $ (U,0) $ in the analytic symmetric case, giving a unique solution to~\eqref{3.1}. \qed

\section{Proof  of Theorem \ref{anasymprep}  in the smooth case}\label{smooth}

The smooth  $ N \times N $ systems form a local Fr\'echet space with the $ C^k(\ol \omega) $ norms denoted by $\mn{\cdot}_k$ with $\mn{\cdot} = \mn{\cdot}_0$, which is tame if $\, \ol \omega $ is a compact neighborhood with boundary, see Theorem~\ref{ThmB.9}. 
For tame Fr\'echet spaces one can use the Nash--Moser theorem, which is the inverse function theorem for these spaces.  The difference to the analytic case is that is it not enough to check the local surjectivity of the differential  when $ \bM \equiv 0 $, instead one has to check that for any $ \bM $ in a neighborhood, see Theorem~\ref{ThmB.10}.

So we shall show that if $ U(t,x) \in C^\infty $, $| U(0,0) | \ne 0$, $ \bM (t,x) = \sum_{j=0}^{m-1} t^j  M_j(x) \in C^\infty $ is symmetric such that $ M_j(0) = 0,\ \forall \, j  $, then
the following $ C^\infty $  mapping is surjective near~$ (0,0) $
\begin{multline}\label{4.1}
	(U(t,x), \bM(t,x)) \overset{\Cal F}{\longrightarrow} \Cal F(U,\bM) = U(t,x)(t^m \id_N + \bM(t,x))U^*(t,x) = F(t,x)
\end{multline}
Here $ F $ is symmetric,  $\partial_t^m F(0,0) > 0 $ and $\partial_t^j F(0,0) =  0$ for $j < m$. 
By Example~\ref{tamemapex} the mapping $ \Cal F $ is tame of degree 0.

When $ \bM \equiv 0 $ we find that $ \partial_t^m(t^mUU^*) = m!UU^*= \partial_t^mF $ when $ t = 0 $ so $ U(t,x) $ must be invertible close to $ (0,0) $. 
Let $U_0(t,x)$ such that $U_0U_0^*= \partial_t^mF/m!$ at $(0,0)$ and let $\bM_0(t,x)$ be a symmetric polynomial system of degree less than $m$ such that $\bM_0(t,0) \equiv 0$. 
Observe that we may choose symmetric $U_0$ such that $U_0(0,0) = \partial_t^mF(0,0)^{1/2} > 0$ and then $U_0 > 0$ near  $(0,0)$.
Since we shall only need local solutions, we may assume that $ U_0(t,x) \in C^\infty(\ol {\omega})$ where $ \ol \omega $ is compact neighborhood of  $(0,0)$  so that $ U_0(t,x) $ is invertible and that $ \bM_0(t,x) $ is symmetric such that $\partial_t^m \bM_0(t,x) \equiv 0 $ and $\bM_0(t,0) \equiv 0  $. 

Then we let $ V $ be a bounded neighborhood of $ (U_0(t,x), \bM_0(t,x)) $ in $C^\infty(\ol \omega)$ so that the first component is invertible and the second component is symmetric polynomial in~$ t $ of degree less than~$ m $. 
If $U_0 > 0$ then $V$ is a neighborhood with symmetric matrices.
Since $ \bM_0(t,0) \equiv 0$  we may for $ \varepsilon > 0 $ choose $ V_\varepsilon $ and a compact neighborhood~$ \ol \omega_\varepsilon  $ of  $(0,0)$ so that $ (U, \bM) \in V_\varepsilon  $ gives $ \mnorm{\bM} \le \varepsilon $ in $ \ol \omega_\varepsilon  $.

We find for any $( U(t,x), \bM(t,x)) \in V $  that the differential of $ \Cal F $ at $(U(t,x), \bM(t,x)) $ is equal to
\begin{multline}\label{4.2}
	({\bf u}(t,x), {\bm}(t,x)) \overset{d\Cal F}{\longrightarrow} d \Cal F({\bf u(t,x)},{\bf m}(t,x))\\ = 2\re U(t,x)(t^m\id_N + \bM(t,x)){\bf u}^*(t,x) + U(t,x){\bm}(t,x) U^*(t,x)  = A(t,x)
\end{multline}
where  $ \bu(t,x) $ and $ \bm(t,x) $ are smooth such that $ \bm $ is symmetric and $\partial_t^m \bm \equiv 0$  which gives smooth symmetric $ A $. We shall need the following result.

\begin{prop}\label{dFprop}
	Let $(U(t,x),\bM(t,x)) \in V$ 
	so that \/$| U(t,x) | \ne 0$, $\partial^m_t \bM(t,x) \equiv 0$ and  $ \mnorm{\bM(t,x)} \ll 1 $. Then in a sufficiently small compact neighborhood  $\ol \omega $ of \/$(0,0)$, we find that the mapping $d\Cal F$ given by~\eqref{4.2} is surjective, having a tame right inverse in $\ol \omega $. If $U(t,x)$ is positive definite, then $d\Cal F$ is bijective and has a tame inverse. 
\end{prop}

By Theorem \ref{ThmB.10}, we then obtain that the $ C^\infty $ tame mapping\/ $ \Cal F $ is surjective on a neighborhood $W$ of a given $F(t,x)$ near  $(0,0)$, having a smooth tame right inverse $ \Cal F^{-1} $ giving 
\begin{equation}\label{tameinv}
	(U(t,x), \bM(t,x)) = \Cal F^{-1}(F(t,x)) 
\end{equation}
in a neighborhood of $ (0,0) $.
As in the analytic case, $\im  \Cal F $ will also contain some smooth $ F(t,x)$ that does not vanish at $(0,0)$.

If  $U(t,x)$ is positive definite, then by Theorem \ref{ThmB.10} we obtain that the mapping $ \Cal F $  is a bijection in a neighborhood $W$ of a given $F(t,x)$ near  $(0,0)$  with the tame inverse  $ \Cal F^{-1} $. This gives a unique symmetric solution to~\eqref{tameinv} in the smooth case and finishes the proof of Theorem~\ref{anasymprep}. \qed

\begin{proof}[Proof of Proposition~\ref{dFprop}]
We obtain from~\eqref{4.2} that
\begin{multline}
	2 \re t^m U(t,x) \bu^*(t,x) +  U(t,x) \bm(t,x) U^*(t,x) \\ = A(t,x)  - 2 \re\big (U(t,x)\bM(t,x) \bu^*(t,x)\big)
\end{multline}
and by Remark~\ref{Uniqm} the equation 
\begin{equation}\label{TmA}
	T_m(U(t,x) \bm(t,x) U^*(t,x))  = T_m\left(A(t,x)  - 2 \re\big (U(t,x)\bM(t,x) \bu^*(t,x)\big)\right)
\end{equation}
has the the unique symmetric polynomial solution 
\begin{equation}\label{meq}
\bm = T_m\left(U^{-1}A(U^{-1})^* - 2 \re\big (\bM \bu^*(U^{-1})^*\big)\right) 
\end{equation}
which depend linearly and tamely of degree $m-1$ on $ \bM \bu^*$.
Since $\bM(t,x) = \sum_{j= 0}^{m-1}t^j M_j(x)$ and  $ \mnorm{T_m\bu^* }_k \le \mnorm{\bu^* }_{k+m-1} $ we find
\begin{equation}\label{Tmest}
	\mnorm{ T_m(\bM \bu^*)}_{k}  \ls \mnorm{\bM }  \mnorm{ \bu^*}_{k+m-1} \qquad \forall\ k \ge 0
\end{equation}
where $\mn{\bM} = \sum_{j= 0}^{m-1}\mn{M_j}$.

Then we find from~\eqref{TmA} that
\begin{multline}\label{symeq}
2 \re U(t,x) \bu^*(t,x) = \big(A(t,x) - U(t,x) \bm(t,x) U^*(t,x)  \\ - 2 \re\big (U(t,x)\bM(t,x) \bu^*(t,x)\big)\big)/t^m = A_0(t,x) \in C^\infty
\end{multline}
Observe that $ A_0 $ is symmetric and depends linearly on $ \bm $ and thus linearly on $ T_m(\bM \bu^*) $  by \eqref{meq}.
By \eqref{symeq} there exists a smooth antisymmetric $ B_0(t,x)$ so that
\begin{equation}\label{4.5}
2\bu^*(t,x)  =  U^{-1}(t,x)\big( A_0(t,x)  + B_0(t,x) \big)
\end{equation}
The equations~\eqref{meq} and~\eqref{4.5} gives a linear system of equations in $\bu^*$ and $\bm$, depending linearly on $ \bM \bu^*$ and $B_0$.
We shall keep $ B_0(t,x) $ fixed in the following, observe that we may take any $B_0(t,x)$, for example $ B_0(t,x) \equiv 0 $, which gives different solutions.

Thus for fixed $ U $, $ \bM $ and $B_0$  we obtain
\begin{equation}\label{5.11}
\left\{
\begin{aligned}
&\bu^* = U_{00} +  U_{01}\bM \bu^* + U_{02}\bm + U^{-1}B_0\\
&\bm = M_{00} +  M_{01}\bM \bu^*
\end{aligned}
\right.
\end{equation}
Here the products, except $U^{-1}B_0$, are not matrix multiplications, they represent  linear maps and depend on $U$. We have that $M_{00}$ is symmetric polynomial of degree less than~$m$ and that $M_{01}$ maps smooth functions to such polynomials. Because of~\eqref{meq}, this linear system is tame of degree $m-1$.
Then for fixed $ U $, $\bM$ and $B_0$ we find that~\eqref{4.2} holds if
\begin{equation}\label{5.12}
\left\{
\begin{aligned}
&\bu^* - U_{01}\bM \bu^* -  U_{02}\bm = U_{00} + U^{-1}B_0 \\
&\bm -  M_{01}\bM \bu^* = M_{00} 
\end{aligned}
\right.
\end{equation}

Then we get from~\eqref{Tmest} that~\eqref{5.12} is given by
\begin{equation}\label{approxsys}
\begin{pmatrix}
\id_N & -U_{0,2}   \\
0 & \id_N 
\end{pmatrix}
\end{equation}
modulo terms having $ C^k $ norms that are $ \Cal O( \mnorm{\bM }) $, $\forall \, k$.
Thus, the system~\eqref{5.12} is invertible in a compact neighborhood~$ \ol \omega$ of  $(0,0)$  if $ \mnorm{\bM } \ll 1$, which can be obtained by shrinking $ V $ and $ \ol \omega $ because $M_j(0) = 0$, $\forall \, j$. Since this system is linear, tame and invertible, we find from Example~\ref{tamemapex} that it has a tame inverse  which for any given~$B_0$ gives a smooth solution $( \bu^*, \bm)$ in $\ol  \omega $. 
This gives surjectivity and a tame right inverse to $ d \Cal F $ in~$\ol \omega $  for any $ (U(t,x),\bM(t,x)) \in V $. 
Observe that the neighborhood~$\ol \omega $ only depends on the bounds on $ U $ and $ \bM  $.

It remains to prove uniqueness in the case when $ U_0 > 0$. Then $U > 0$ if $ (U,\bM) \in V $ and $V$ is a small enough neighborhood of $(U_0, \bM_0)$ with symmetric systems. 
We shall show that there is a unique symmetric solution $(\bu^*, \bm)$ to~\eqref{5.12} for some $ B_0(t,x) $ in~\eqref{4.5}. 
By adding the equation $ \im \bu^* \equiv 0$ to~\eqref{5.12} with $\bu^* $ given by~\eqref{5.11}, we get the system
\begin{equation}\label{5.13}
\left\{
\begin{aligned}
	&\bu^* - U_{01}\bM \bu^* -  U_{02}\bm -  U^{-1}B_0 = U_{00}  \\
	&\bm -  M_{01}\bM \bu^* = M_{00}\\
	&\im (U_{01}\bM \bu^*) +  \im( U_{02}\bm ) + \im( U^{-1}B_0 )= - \im U_{00}  
\end{aligned}
\right.	
\end{equation}
As before, the matrix of the system is  
\begin{equation}\label{5.14}
\begin{pmatrix}
\id_N & - U_{02} &  - U^{-1} \\
0 & \id_N  & 0  \\
0  & \im(U_{02}\ \boldsymbol{\cdot}) & \im(U^{-1}\ \boldsymbol{\cdot})
\end{pmatrix}
\end{equation}
modulo terms having $ C^k $ norms that are  $ \Cal O( \mnorm{\bM }) $, $\forall \, k$.
Here $\im(U^{-1}\ \boldsymbol{\cdot})$ is invertible on antisymmetric matrices by Lemma~\ref{symlemma}, which gives that~\eqref{5.14} is invertible. 

Thus, we find that~\eqref{5.13} is invertible in a sufficiently small compact neighborhood~$\ol  \omega $ of~$(0,0)$  if $ \mnorm{\bM } \ll 1 $, which can be obtained by shrinking $ V $  and $ \ol \omega $. 
Thus, we obtain a unique smooth solution $ (\bu, \bm,B_0) $ to~\eqref{5.13} with symmetric $\bu$ and $\bm$ such that $\partial_t^m \bm \equiv 0$, and antisymmetric $B_0$. 
So $ d \Cal F $ is a tame bijection in a compact neighborhood $\ol  \omega $ of  $(0,0)$  for any symmetric $ (U(t,x),\bM(t,x)) \in V $ such that $U(t,x) > 0$, with a tame inverse $d \Cal F^{-1}$ in $\ol  \omega $. As before, the neighborhood  $\ol \omega $ only depends on $ U $ and~$ \bM $. This finishes the proof of Proposition~\ref{dFprop}.
\end{proof}

\appendix

\section{A symmetry lemma}

In this section, we are going to prove the following result about symmetric and antisymmetric invariance of matrices.

\begin{lem}\label{symlemma}
	If $ U(x) > 0  $ is a smooth symmetric $ N \times N $ matrix, then $ A(x) \mapsto \re(UA)(x) $ is a bijection of smooth  symmetric matrices and  $ A(x) \mapsto \im(UA)(x) $ is a bijection of smooth  antisymmetric matrices. If $ U(x) $ is analytic, then this holds for analytic $ A(x) $.
\end{lem}

\begin{proof}
	First we consider the case of constant matrices. To simplify, we define the symmetric product $ S(U,A) = \frac{1}{2}(UA + AU) $, which is symmetric when $ A$ is symmetric and antisymmetric when $ A $ is antisymmetric, thus $A \mapsto S(U,A) $ preserves the symmetry. 
	We also have that linear ON mappings $ O $ preserves the symmetry, since $ O A O^{-1} = OAO^* $.
	
	Next, we are going to prove that $A \mapsto S(U,A)$ is injective for both symmmetric and antisymmetric $A$.
	Since any $ N \times N $ matrix can be written uniquely as a sum of a symmetric and an antisymmetric matrix, we will then find that $ A \to S(U,A) $ is injective for $ N \times N $ matrices $ A $.
	Thus, by using an ON base of eigenvectors to $ U $, we find that $ U = \{ \mu_j \delta_{jk}\}_{jk} $ with $ \mu_j > 0, \ \forall \, j $, and $ A = \{ \alpha_{jk}\}_{jk} $. If $ A $ is symmetric then  $  \alpha_{kj} = \alpha_{jk} $ $\forall \, j k $, and if $ A $ is antisymmetric then  $  \alpha_{kj} = - \alpha_{jk} $   $\forall \, j k $. In both cases, we find that 
	$$ 
	S(U,A)  =  \{ (\mu_{j} + \mu_{k})\alpha_{jk}/2 \}_{jk} = 0
	$$
	if and only if $ \alpha_{jk} = 0 $, $\forall \, j k $. So $ A \mapsto S(U,A) $ is injective, and thus bijective in both cases. Observe that $ |\alpha_{jk} | \le \mn{U^{-1}} \mn{S(U,A)}$ since $ \mu_j^{-1} \le \mn{U^{-1}} $,  $ \forall \, j $.
	
	Next, we consider  the equation
	$$ 
	S(U,A)(x) = S(U(x),A(x)) = B(x) \in C^\infty
	$$
	which has a unique solution $A(x) = \{ \alpha_{jk}\}_{jk}(x) $ such that $  |\alpha_{jk}(x) | \le \mn{U^{-1}(x)} \mn{B(x)} $. By compactness we find for any $ x_0 $ and any sequence $ x_j \to x_0 $ that there exist a subsequence having a limit: $ A(x_j) \to A_0$. Then by continuity and uniqueness we find that $ \lim_{x \to x_0} A(x) = A_0 = A(x_0)$, which proves continuity of $ A(x) $.
	
	To show that $ A(x) $ is differentiable, we put $ \Delta_{h,y}A(x) = \frac{A(x + hy) - A(x)}{h}$ for any $ x, \, y \in \br^n$ and $ h > 0 $. Then we find
	\begin{equation}
		\Delta_{h,y}S(U,A)(x) = S(\Delta_{h,y} U(x),A(x+ hy)) + S(U(x),\Delta_{h,y}A(x)) = \Delta_{h,y}B(x)
	\end{equation}
	which gives that 
	\begin{equation}
		S(U(x),\Delta_{h,y}A(x)) = \Delta_{h,y}B(x) - S(\Delta_{h,y} U(x),A(x+ hy)) 
	\end{equation}
	where the right hand side is continuous in $ h $. As before, we that the left hand side has the unique limit $ S(U(x_0), y\cdot \partial A(x_0)) $ for any $ y $, giving $ S(U, \partial A) =  \partial B  - S(\partial U, A)$.
	
	Let $ A_\alpha(x) = \partial^\alpha A(x) $ and $ A_k(x) = \{ A_\alpha (x)\}_{|\alpha | \le k} $ and similarly for $ U(x) $ and $ B(x) $. 
	By induction, we can assume that $ A(x) \in C^k$ and
	that 
	\begin{equation}
		B_\gamma(x) = \sum_{\alpha + \beta = \gamma}c^\gamma_{\alpha \beta}S(U_\alpha(x) ,A_\beta(x)
		)  \qquad \text{for $ | \gamma | = k $}
	\end{equation}
	Then we find for $ |\beta | = k $ that
	\begin{multline}
		S(U(x),\Delta_{h,y}A_\beta(x)) = \Delta_{h,y}B_\beta(x) - S(\Delta_{h,y} U(x),A_\beta(x+ hy)) \\ - \Delta_{h,y}\sum_{\substack{\alpha + \beta = \gamma \\ \alpha \ne 0}}c_{\alpha\beta}S(U_\alpha(x),A_{\beta - \alpha}(x))
	\end{multline}
	where the right hand side is continuous in $ h $. As before, we find that  that the left hand side has the limit $ S(U,y \cdot \partial A_{k})(x) $ which gives that $ A_k(x) \in C^1 $ so $ A(x) \in C^{k+1}$. By induction we obtain the result in the $C^\infty$ case.
	
	If $ U(x) $ and $ B(x) $ are analytic, then they have holomorphic extensions in a complex neighborhood of $ \br^n $. We then find that $ S(U(z), \partial_{\bar z_j}A(z)) =  \partial_{\bar z_j}B(z)  - S(\partial_{\bar z_j}U(z), A(z)) = 0$ in that neighborhood $ \forall \, j $, which gives the analyticity of $ A(x) $ by uniqueness.
\end{proof}

\section{Tame Fr\'echet spaces and the Nash--Moser theorem}\label{tamenm}

In this section, we will recall the Nash--Moser Theorem and the theory of tame Fr\'echet spaces and tame maps from Sections II and III in \cite{ham}. Recall that a Fr\'echet space is a convex and complete vector space with topology defined by a countable number of seminorms, thus it is a complete metrizable vector space. But for the Nash--Moser Theorem one needs Fr\'echet spaces with a more refined structure.

\begin{defn}
	A grading on a Fr\'echet space is a collection of seminorms $ \{ \mn{\cdot }_n , \ n \in \bn\} $ that defines the topology and is increasing: $ \mn{ f}_j \le  \mn{ f}_k $ when $ j \le k $.
\end{defn}

Of course, every Fr\'echet space has a grading, for example by adding to each seminorm all the lower ones. Examples are $ C^{\infty} $ functions and vector bundles with $ C^k $ norms. Closed subspaces of graded spaces are graded with the induced norms, and Cartesian products $ F \times G $ can be graded with the norms $\mn{ (f,g)}_n = \mn{f}_n  + \mn{g}_n$. 

\begin{defn}
	Two gradings $ \mn{\cdot }_n$ and $ \mn{\cdot }'_n$ are tamely equivalent of degree $ r $ and base~$ b $ if
	\begin{equation}
	 \mn{f }_n \le C_n \mn{f }'_{n+r} \quad\text{and} \quad \mn{f }'_n \le C'_n \mn{f }_{n+r}
  	\end{equation}
  for any $ n \ge b $. If this holds for any $ n $ then they are equivalent of degree $ r $.
\end{defn} 

For example, if $ X $ is a compact manifold, then the $ C^k $ norms, H\"older norms and Sobolev norms are tamely equivalent.

\begin{defn}
	A linear map $ L: \ F \mapsto G $ of a graded space into another is tame, if it satisfies a tame estimate of degree $ r $ and base $ b $:
	\begin{equation}\label{B.2}
		\mn{Lf }_n \le C_n \mn{f }_{n+r}
	\end{equation}
for any $ n \ge b $. If\/ \eqref{B.2} holds for any $ n $ then $ L $ is tame  of degree $ r $.
\end{defn}

\begin{exe}\label{tamemapex}
	A $ C^\infty $  linear partial differential operator $ L $  of order $ m $  on $ C^\infty $ is a tame of degree $ m $ since $ \mn{Lf }_n \le C_n \mn{f }_{n+m} $ for the $ C^k $ norms for any $ n $. 
	In particular, multiplication with  $ C^\infty $ functions are tame of degree 0. 
	Solutions to invertible linear systems of smooth equations give tame maps of degree 0 by Leibniz' rule.
	Similarly, for any $ m \in \bn $ and $ C^\infty $ systems $U(t,x)$ and $ \bM(t,x) $, we find that the mapping 
	$$ 
	(U(t,x), \bM(t,x)) \mapsto F(t,x) = U(t,x)(t^m\id_n + \bM(t,x)) U^*(t,x)
	$$
	is tame of degree 0.
\end{exe}

\begin{exe}
	The linear mapping $L: C^{\infty}[0,1] \mapsto  C^{\infty}[-1,1] $ given by $ Lf(x) = f(x^2) $ is tame with $ \mn{Lf}_n \le C_n \mn{f}_n $. The image of $ L $ is the closed subspace of even functions $  C_S^{\infty}[-1,1]  $. But the inverse $ L^{-1}:\  C_S^{\infty}[-1,1] \mapsto  C^{\infty}[0,1]$ is not tame, since the $ 2n $:th  Taylor coefficient of $ Lf = g $ is the $ n $:th Taylor coefficient of $ f = L^{-1}g $ so the best estimate is $ \mn{L^{-1}f}_n \le C_n \mn{f}_{2n} $.
\end{exe}

Compositions of linear tame maps are linear and tame, and $ L $ is a tame isomorphism if $ L $ is linear isomorphism such that $ L $ and $ L^{-1} $ are tame.

\begin{defn}
	Let $ F $ and $ G $ be graded spaces. Then we say that $ F $ is a tame direct summand of $ G $ if we can find tame linear maps $ L: \ F \mapsto G $ and $ M: \ G \mapsto F $ so that the composition $ ML : \ F \mapsto F $ is the identity: $F \overset{L}{\mapsto} G \overset M\mapsto F  $.
\end{defn}

\begin{defn}
	We say that a graded space is tame if it is a tame direct summand of a space $ \Sigma (B) $ of exponentially decreasing sequences in some Banach space $ B $.
\end{defn}

A Cartesian product of two tame spaces is tame. 

\begin{thm}\label{ThmB.9}
The space $ C^\infty(X) $ is tame if $ X $ is a compact manifold (with or without boundary), and if $ V $ is a  vector bundle over $ X $ then the sections of the vector bundle $ C^\infty(X,V) $ is also tame, since it can be written as a tame direct summand of a tame space. 
\end{thm}

\begin{thm}
If $ X $ is a compact manifold with boundary, then $ C^\infty_0(X) $ and $ C^\infty_0(X,V) $ are tame. 
\end{thm}

Next, we shall look at nonlinear maps.

\begin{defn}
Let $ F $ and $ G $ be graded spaces and $ P: \ F \supseteq U \mapsto G $ a nonlinear map, then we say that $ P $ satisfies a tame estimate of degree $ r $ and base $ b $ if
\begin{equation}\label{B.3}
\mn{P(f)}_n \le C_n(1 + \mn{f}_{n+r})
\end{equation}
for all $ f \in U  $ and all $ n \ge b $. If \eqref{B.3} holds for any $ n $ then $ P $ is tame of degree $ r $. 
We say that $ P $ is a tame map if $ P $ is defined on an open set, is continuous and satisfies a tame estimate in a neighborhood of any point. If \/ \eqref{B.3} holds for any $ n $ then we say that $ P $ satisfies a tame estimate of degree $ r $.
\end{defn}

\begin{thm}
A composition of tame maps is tame.
\end{thm}

The following is the Nash--Moser theorem that we shall use, which is Theorem 1.7.1 in section III of \cite{ham}.

\begin{thm}\label{ThmB.10}
Let $ F $ and $ G $ be tame spaces and $ P:\ F\supseteq U \mapsto G $ a $ C^\infty $  tame map of the open set\ $ U $. If the equation for the differential $ DP(f)h = k $ has a unique solution $ h = VP(f)k $  for any $ f \in U $ and any $ k $, and the family of inverses $VP(f): \ U \times G \mapsto F  $ are tame maps, then $ P $ is locally invertible and each local inverse is a $ C^\infty $ tame map.

If the derivative $ DP(f)h = k $ has a solution $ h = VP(f)k $ for any $ f \in U $ and any $ k $, and the family of right inverses $VP(f): \ U \times G \mapsto F  $ are tame maps, then $ P $ is locally surjective and each local right inverse is a $ C^\infty $ tame map.

If the derivative $ DP(f) $ is injective for any $ f \in U$ and the family of left inverses $VP(f): \ U \times G \mapsto F  $ are tame maps, then $ P $ is locally injective and each local left inverse is a $ C^\infty $ tame map.
\end{thm}

The proofs of these results and more examples and counterexamples can be found in sections II and III in \cite{ham}.

\begin{ack}
	I would like to thank Gregory Berkolaiko at Texas A\&M, who asked me about symmetric preparation theorems like \eqref{sprepconj} which started this project. We then worked on the formal power series solution in the first order case. Gregory has also provided very useful comments and improvements of the manuscript.  
\end{ack}

\bibliographystyle{plain}

	\end{document}